\documentclass{amsart}
\usepackage{hyperref, url}               
\hypersetup{bookmarksdepth=subsubsection}
\usepackage[english]{babel}         
\usepackage[utf8]{inputenc}         

\usepackage{csquotes}               
\usepackage{amssymb, amsmath, amsthm}
\usepackage{dsfont}                 
\usepackage{tikz-cd, tikz, mathtools}
\usepackage{enumerate}			    
\usepackage[capitalise]{cleveref}   
\usepackage{framed} 		        

\usepackage[alphabetic]{amsrefs}

\DeclareMathOperator{\AdTrip}{AdTrip}

\newcommand{\cart}{\mathrm{cart}}

\newcommand{\Cocart}{\mathrm{Cocart}}

\newcommand{\bCocart}{\mathbf{Cocart}}
\newcommand{\bCart}{\mathbf{Cart}}
\newcommand{\Str}{\mathrm{Str}}
\newcommand{\Un}{\mathrm{Un}}
\mathchardef\dash="2D

\newcommand{\bCat}{\mathbf{Cat}}
\newcommand{\bSpan}{\mathbf{Span}}
\newcommand{\bFun}{\mathbf{Fun}}

\newcommand{\co}{\mathrm{co}}
\newcommand{\coop}{\mathrm{co}\,\mathrm{op}}
\DeclareMathOperator{\BC}{BC}
\newcommand{\ladj}{\textup{-ladj}}
\newcommand{\radj}{\textup{-radj}}
\newcommand{\cc}{\textup{cc}}
\newcommand{\ct}{\textup{ct}}
\newcommand{\Unf}{\mathrm{Unf}}
\newcommand{\bUnf}{\mathbf{Unf}}
\newcommand{\cov}{\textup{cov}}
\newcommand{\contra}{\textup{con}}

\DeclareMathOperator{\Hom}{Hom}

\DeclareMathOperator{\Fun}{Fun}

\DeclareMathOperator{\id}{id}
\DeclareMathOperator{\pr}{pr}

\let\lim\relax
\DeclareMathOperator{\lim}{lim}

\DeclareMathOperator{\PSh}{PSh}

\DeclareMathOperator{\Span}{Span}

\newcommand{\catop}{^{\textup{op}}}
\newcommand{\catco}{^{\textup{co}}}

\newcommand{\catname}[1]{{\textup{#1}}}

\newcommand{\Cat}{\catname{Cat}}

\newcommand{\qquadtext}[1]{\qquad\textrm{#1}\qquad}

\let\op\relax
\DeclareMathOperator{\op}{op}

\renewcommand{\phi}{\varphi}
\renewcommand{\epsilon}{\varepsilon}

\newcommand{\Cc}{\mathcal{C}}
\newcommand{\Dd}{\mathcal{D}}
\newcommand{\Ee}{\mathcal{E}}

\theoremstyle{plain}
\newtheorem*{mainthm}{Main Theorem}
\newtheorem{theorem}{Theorem}[section]

\newtheorem{corollary}[theorem]{Corollary}
\newtheorem{lemma}[theorem]{Lemma}
\newtheorem{proposition}[theorem]{Proposition}

\newtheorem*{theorem*}{Theorem}
\newtheorem*{corollary*}{Corollary}

\theoremstyle{definition}

\newtheorem{construction}[theorem]{Construction}

\newtheorem{definition}[theorem]{Definition}

\newtheorem{example}[theorem]{Example}
\newtheorem{notation}[theorem]{Notation}
\newtheorem{observation}[theorem]{Observation}

\newtheorem{remark}[theorem]{Remark}

\newtheorem*{remark*}{Remark}

\newcommand{\iso}{\xrightarrow{\;\smash{\raisebox{-0.5ex}{\ensuremath{\scriptstyle\sim}}}\;}}

\tikzcdset{pullback/.style = {phantom, "\lrcorner"{very near start}}}
\tikzcdset{pushout/.style = {phantom, "{\rotatebox{180}{$\lrcorner$}}"{very near end}}}

\usepackage{microtype}

\title[Universality of Barwick's unfurling construction]{Universality of \\Barwick's unfurling construction}
\author{Bastiaan Cnossen}
\address{B.C.: Fakultät für Mathematik, Universität Regensburg, 93040 Regensburg, Germany}
\author{Tobias Lenz}
\address{T.L.: Mathematisches Institut, Rheinische Friedrich-Wilhelms-Universität Bonn, Endenicher Allee 60, 53115 Bonn, Germany}
\author{Maxime Ramzi}
\address{M.R.: FB Mathematik und Informatik, Universität Münster, Einsteinstraße 62, 48149 Münster, Germany }

\begin{document}
\begin{abstract}
	Given an $\infty$-category $\Cc$ with pullbacks, its $(\infty,2)$-category $\bSpan(\Cc)$ of spans has the universal property of freely adding right adjoints to morphisms in $\Cc$ satisfying a Beck--Chevalley condition. We show that this universal property is implemented by an $(\infty,2)$-categorical refinement of Barwick's \emph{unfurling construction}: For any right adjointable functor $\Cc \to \Cat_{\infty}$, the unstraightening of its unique extension to $\bSpan(\Cc)$ can be explicitly written down as another span $(\infty,2)$-category, and on underlying $(\infty,1)$-categories this recovers Barwick's construction.

	As an application, we show that the constructions of cartesian normed structures by Nardin--Shah and Cnossen--Haugseng--Lenz--Linskens coincide.
\end{abstract}
\maketitle

\section{Introduction}

In recent years, \textit{span categories}, also known as \emph{categories of correspondences}, have been widely used across various areas of pure mathematics to encode assignments that admit both covariant and contravariant functoriality that interact via a Beck--Chevalley property. In particular, they appear in the study of topological field theories \cites{Lurie2009ClassificationTFTs,HaugsengSpans}, motivic homotopy theory \cites{Voevodsky_Correspondence, CisinskiDeglise2019Triangulated,EHKSY_Algebraic_Cobordism}, equivariant homotopy theory \cites{nardin2016exposeIV, Barwick2017SpectralMackey, BalmerDellAmbrogio2020Mackey, Kaledin_Mackey_Profunctors, Lenz-Mackey, marc-n-infty}, six-functor formalisms \cites{GaitsgoryRozenbluym2017StudyDAG,mann2022sixFunctor,ScholteSixFunctors, Kuijper_Six_Functors}, higher commutative/normed structures \cites{harpaz2020ambidexterity, BachmannHoyois2021Norms,CLL_Spans, CHLL_Bispans2}, and the theory of traces \cites{BenZviNadler, CCRY_Characters}, among others.

An important feature of span categories is that they naturally encode adjunctions: given a functor $F\colon \Cc \to \Cat_{\infty}$ satisfying certain adjointability conditions, there exists an extension $\Span(\Cc) \to \Cat_{\infty}$ which captures both the original functoriality and a compatible system of adjoints. Two approaches to making this precise have emerged in the literature:
\begin{itemize}
    \item The first, due to Barwick \cite{Barwick2017SpectralMackey}, provides an explicit construction called \textit{unfurling}: applying the span construction to the cartesian fibration $p\colon \Ee \to \Cc$ classified by $F$ gives rise to a cocartesian fibration $\Span(\Ee,\Ee^{p\textup{-}\cart},\Ee) \to \Span(\Cc)$ whose unstraightening encodes both $F$ as well as its additional adjoints; here $\Ee^{p\textup{-}\cart} \subseteq \Ee$ denotes the wide subcategory of $p$-cartesian morphisms. 
    \item The second approach, developed by Gaitsgory--Rozenblyum \cite{GaitsgoryRozenbluym2017StudyDAG}, takes a more conceptual view: they showed that $\Span(\Cc)$ can naturally be enhanced to an $(\infty,2)$-category $\bSpan(\Cc)$ with the universal property of freely adding right adjoints to morphisms in $\Cc$. This universal property, later established rigorously by MacPherson \cite{MacPherson2022Bivariant} and Stefanich \cite{Stefanich2020Correspondences}, ensures that any right adjointable functor $F\colon \Cc \to \Cat_{\infty}$ extends uniquely to a 2-functor $\bSpan(\Cc) \to \bCat_{\infty}$.
\end{itemize}

While these approaches should give equivalent results, their relationship has not been made explicit. In this paper, we show that Barwick's unfurling construction can be enhanced to produce precisely the 2-functor coming from the universal property:
\begin{mainthm}[\Cref{thm:Universality_Unfurling}]
	Let $\Cc$ be an $\infty$-category with pullbacks and let $F\colon \Cc \to \bCat_{\infty}$ be a right adjointable functor. Then the 2-functor $\bSpan(\Cc) \to \bCat_{\infty}$ obtained from $F$ via the universal property is equivalent to the 2-categorical unfurling construction applied to $F$.
\end{mainthm}

More precisely, we show that the 2-functor
\[
\bSpan(p)\colon \bSpan(\Ee,\Ee^{p\textup{-}\cart},\Ee) \to \bSpan(\Cc)
\]
induced by the cartesian fibration $p\colon \Ee \to \Cc$ classified by $F$ is a 1-cocartesian fibration of $(\infty,2)$-categories whose pullback along $\Cc \hookrightarrow \bSpan(\Cc)$ recovers $p$, refining Barwick's result that this holds on the underlying $(\infty,1)$-categories. By 2-categorical straightening/unstraightening, recalled in \Cref{subsec:2-categorical-unstraightening}, this results in a 2-functor
\[
    \bUnf(F)\colon \bSpan(\Cc) \to \bCat_{\infty}
\]
extending $F$, which by uniqueness must be the universal one.

We further prove several variations of this result:
\begin{itemize}
    \item A version for functors $F\colon \Cc \to \Cat_{\infty}$ that only admit right adjoints for \textit{some} morphisms in $\Cc$;
    \item A version for \textit{contravariant} \textit{left} adjointable functors $F\colon \Cc\catop \to \Cat_{\infty}$;
    \item A version for contravariant \textit{right} adjointable functors $F\colon \Cc\catop \to \Cat_{\infty}$.
\end{itemize}
The first two variations are proved in an almost identical way to our main result, and all these variations are captured by \Cref{thm:Universality_Unfurling}. The third variation is a bit more subtle and is addressed in \Cref{prop:subtle}.

The uniqueness part of the $(\infty,2)$-universal property sometimes allows one to easily identify different unfurling constructions where comparing the underlying $(\infty,1)$-categories by hand would be cumbersome. As an example of this, we prove that the two competing constructions of \emph{cartesian normed structures} due to Nardin--Shah \cite{NardinShah} and to Cnossen--Haugseng--Lenz--Linskens \cite{CHLL_Bispans2} agree, affirming a conjecture of the latter authors.

\subsection*{Conventions} We freely use the language of $\infty$-categories, as developed in \cite{lurie2009HTT,lurie2016HA}. By an \textit{$(\infty,2)$-category} we mean a 2-fold complete Segal space. We denote an $(\infty,2)$-categorical enhancement of an $(\infty,1)$-category with a boldface font, e.g.\ $\bSpan(\Cc)$ and $\bCat_{\infty}$. For an $(\infty,2)$-category $\Dd$, we denote by $\Dd\catop$ the $(\infty,2)$-category obtained from $\Dd$ by flipping the directions of 1-morphisms, and by $\Dd\catco$ the one obtained by flipping the directions of 2-morphisms.

\subsection*{Acknowledgments} B.C.\ is an associate member of the SFB 1085 `Higher Invariants' at the University of Regensburg, funded by the DFG. T.L.\ is an associate member of the Hausdorff Center for Mathematics at the University of Bonn (DFG GZ 2047/1, project ID 390685813). M.R.\ is funded by the Deutsche Forschungsgemeinschaft (DFG, German Research Foundation) -- Project-ID 427320536 -- SFB 1442, as well as by Germany's Excellence Strategy EXC 2044 390685587, Mathematics Münster: Dynamics--Geometry--Structure, and in the beginning stages of this article was supported by the Danish National Research Foundation through the Copenhagen Centre for Geometry and Topology (DNRF151). 

M.R.\ and B.C.\ are grateful to the University of Copenhagen for hosting a Masterclass on parametrized homotopy theory during which parts of this research were conducted.

\section{The universal property of the span 2-category}

We start by recalling the $(\infty,2)$-category $\bSpan(\Cc,\Cc_L,\Cc_R)$ for an adequate triple $(\Cc,\Cc_L,\Cc_R)$, together with its various universal properties when $\Cc_L = \Cc$ or $\Cc_R = \Cc$.

\begin{definition}
	An \textit{adequate triple} is a triple $(\Cc,\Cc_L,\Cc_R)$ where $\Cc$ is an $\infty$-category and $\Cc_L,\Cc_R \subseteq \Cc$ are wide subcategories such that for every morphism $l\colon X \to Z$ in $\Cc_L$ and $r\colon Y \to Z$ there exists a pullback square in $\Cc$ of the form
	\[
	\begin{tikzcd}
		W \dar[swap]{l'} \rar{r'} \drar[pullback] & X \dar{l} \\
		Y \rar{r} & Z\rlap,
	\end{tikzcd}
	\]
	and for every such pullback square the morphism $l'$ is in $\Cc_L$ while $r'$ is in $\Cc_R$. A \textit{morphism of adequate triples} $p\colon (\Ee,\Ee_L,\Ee_R) \to (\Cc,\Cc_L,\Cc_R)$ is a functor $p\colon \Ee \to \Cc$ satisfying $p(\Ee_L) \subseteq \Cc_L$ and $p(\Ee_R) \subseteq \Cc_R$, and which preserves pullback squares of morphisms in $\Ee_L$ along morphisms in $\Ee_R$.

	A pair $(\Cc,\Cc_L)$ consisting of an $\infty$-category $\Cc$ and a wide subcategory $\Cc_L \subseteq \Cc$ is called a \textit{span pair} if the triple $(\Cc,\Cc_L,\Cc)$ is an adequate triple.
\end{definition}

Every adequate triple $(\Cc,\Cc_L, \Cc_R)$ determines an $(\infty,2)$-category $\bSpan(\Cc,\Cc_L,\Cc_R)$. Informally, its objects are the objects of $\Cc$, its 1-morphisms are spans $X \xleftarrow{l} U \xrightarrow{r} Y$ with $l \in \Cc_L$ and $r \in \Cc_R$, and its 2-morphisms are diagrams of the form
\[
\begin{tikzcd}[row sep=tiny, column sep=small]
	& U \dlar[swap]{l} \drar{r} \ar{dd}{f} \\
	X && Y \\
	& U' \ular{l'} \urar[swap]{r'}
\end{tikzcd}
\]
where composition of 1-morphisms is given by forming pullbacks in $\Cc$. For a precise construction of this $(\infty,2)$-category, we refer to \cite[Corollary~3.1.12]{Stefanich2020Correspondences}. Similarly, every morphism of adequate triples $p\colon (\Ee,\Ee_L,\Ee_R) \to (\Cc,\Cc_L,\Cc_R)$ induces a 2-functor $\bSpan(p)\colon \bSpan(\Ee,\Ee_L,\Ee_R) \to \bSpan(\Cc,\Cc_L,\Cc_R)$.

The universal property of the span 2-category of a span \emph{pair} is phrased in terms of a certain adjointability criterion:

\begin{definition}
	\label{def:BC_Functor}
	Let $(\Cc,\Cc_L)$ be a span pair and let $\Dd$ be an $(\infty,2)$-category. A functor $F\colon \Cc \to \Dd$ is said to be \textit{covariantly left $L$-adjointable} if it satisfies the following conditions:
	\begin{enumerate}[(1)]
		\item For every morphism $l\colon X \to Y$ in $\Cc_L$, the functor $l^* \coloneqq  F(l)\colon F(X) \to F(Y)$ admits a left adjoint $l_!\colon F(Y) \to F(X)$;
		\item For every pullback square
		\[
		\begin{tikzcd}
			X' \dar[swap]{k} \rar{h} \drar[pullback] & X \dar{l} \\
			Y' \rar{g} & Y
		\end{tikzcd}
		\]
		in $\Cc$ with $l \in \Cc_L$, the Beck--Chevalley transformation $h_!k^* \Rightarrow l^*g_!$ is an equivalence of morphisms $F(X) \to F(Y')$ in $\Dd$.
	\end{enumerate}
	If $F,G\colon \Cc \to \Dd$ are covariantly left $L$-adjointable functors, we say that a transformation $\alpha\colon F \Rightarrow G$ is \textit{left $L$-adjointable} if for every morphism $l\colon X \to Y$ in $\Cc_L$ the naturality square
	\[
	\begin{tikzcd}
		F(X) \dar[swap]{F(l)} \rar{\alpha(X)} & G(Y) \dar{G(l)} \\
		F(Y) \rar{\alpha(Y)} & G(Y)
	\end{tikzcd}
	\]
	is vertically left adjointable. We let $\bFun_{L\ladj}(\Cc,\Dd) \subseteq \bFun(\Cc,\Dd)$ denote the locally full subcategory spanned by the left $L$-adjointable functors and the left $L$-adjointable natural transformations.

	Dually, we say that $F$ is \textit{covariantly right $L$-adjointable} if the functor $F\catco\colon \Cc = \Cc\catco \to \Dd\catco$ is covariantly left $L$-adjointable, i.e.\ each $l^*$ has a \textit{right} adjoint $l_*$ satisfying Beck--Chevalley. We similarly obtain a notion of right $L$-adjointable transformations $F \Rightarrow G$, resulting in a subcategory $\bFun_{L\radj}(\Cc,\Dd)$.
\end{definition}

The previous definition can immediately be dualized to functors out of $\Cc\catop$:
\begin{definition}
	Let $(\Cc,\Cc_R)$ be a span pair and let $\Dd$ be an $(\infty,2)$-category. A functor $F\colon \Cc\catop \to \Dd$ is said to be \textit{contravariantly left $R$-adjointable} if the functor $F\catop\colon \Cc \to \Dd\catop$ is covariantly right $R$-adjointable, i.e. each $r^*\coloneqq  F(r)$ admits a {left} adjoint $r_!$ {in $\Dd$} satisfying the Beck--Chevalley condition. We say $F$ is \textit{contravariantly right $R$-adjointable} if $F\catop\colon \Cc \to \Dd\catop$ is covariantly left $R$-adjointable. We obtain subcategories
	\[
	\bFun_{R\ladj}(\Cc\catop,\Dd) \subseteq \bFun(\Cc\catop,\Dd) \qquadtext{ and } \bFun_{R\radj}(\Cc\catop,\Dd) \subseteq \bFun(\Cc\catop,\Dd).
	\]
	When the variance is clear from context, we will drop the adverb `co/contravariantly' and simply refer to left/right adjointable functors.
\end{definition}

\begin{remark}
	Our terminology is a variation of that of \cite[Section~2.2]{ElmantoHaugseng2023Distributivity}, who use the phrases `adjointable' and `coadjointable' for what we call `contravariantly adjointable' and `covariantly adjointable,' respectively. Such functors are also called \textit{bivariant theories} \cite[Section~3.2]{MacPherson2022Bivariant} or \textit{functors satisfying the Beck--Chevalley condition} \cite[Definition~3.4.5]{Stefanich2020Correspondences}.
\end{remark}

If $(\Cc,\Cc_L)$ is a span pair, it is not difficult to check that the inclusion $\Cc \hookrightarrow \bSpan(\Cc,\Cc_L,\Cc)$ is covariantly right $L$-adjointable: for a morphism $l\colon X \to Y$, its image in the span 2-category is given by the span $X \xleftarrow{\id_X} X \xrightarrow{l} Y$, and one can observe that a right adjoint in $\bSpan(\Cc,\Cc_L,\Cc)$ is given by the span $Y \xleftarrow{l} X \xrightarrow{\id_X} X$, with unit and counit given by the following two morphisms of spans:
\[
\begin{tikzcd}
	X \rar[equal] & X \dar{\Delta_l} \rar[equal] & X \\
	& X \times_Y X \ular{\pr_1} \urar[swap]{\pr_2}
\end{tikzcd}
\qquadtext{ and }
\begin{tikzcd}
	& X \dlar[swap]{l} \drar{l} \ar[d, "l"] \\
	Y \rar[equal] & Y \rar[equal] & Y.
\end{tikzcd}
\]
It turns out that this inclusion is in fact \textit{universal} among right $L$-adjointable functors out of $\Cc$, see for example \cite[Chapter 7, Theorem~3.2.2]{GaitsgoryRozenbluym2017StudyDAG}, \cite[Theorem~4.2.6]{MacPherson2022Bivariant}, \cite[Theorem~3.4.18]{Stefanich2020Correspondences}.

\begin{theorem}
	Let $(\Cc,\Cc_L)$ be a span pair and let $\Dd$ be an $(\infty,2)$-category. Then the inclusion $\Cc \hookrightarrow \bSpan(\Cc,\Cc_L,\Cc)$ is right $L$-adjointable and restriction along this inclusion induces an equivalence of $(\infty,2)$-categories
	\[
	\bFun_2(\bSpan(\Cc,\Cc_L,\Cc),\Dd) \iso \bFun_{L\radj}(\Cc,\Dd).\pushQED{\qed}\qedhere\popQED
	\]
\end{theorem}

This theorem has the following immediate corollaries:

\begin{corollary}\label{cor:other-univ-properties}
	Let $(\Cc,\Cc_L)$ and $(\Cc,\Cc_R)$ be span pairs, and let $\Dd$ be an $(\infty,2)$-category.
	\begin{enumerate}[(1)]
		\item The inclusion $\Cc \hookrightarrow \bSpan(\Cc,\Cc_L,\Cc)^{\co}$ is the universal covariant left $L$-adjointable functor:
		\[
			\bFun_2(\bSpan(\Cc,\Cc_L,\Cc)^{\co},\Dd) \iso \bFun_{L\ladj}(\Cc,\Dd).
		\]
		\item The inclusion $\Cc\catop \hookrightarrow \bSpan(\Cc,\Cc,\Cc_R)$ is the universal contravariant left $R$-adjointable functor:
		\[
			\bFun_2(\bSpan(\Cc,\Cc,\Cc_R),\Dd) \iso \bFun_{R\ladj}(\Cc\catop,\Dd).
		\]
		\item The inclusion $\Cc\catop \hookrightarrow \bSpan(\Cc,\Cc,\Cc_R)^{\co}$ is the universal contravariant right $R$-adjointable functor:
		\[
		\bFun_2(\bSpan(\Cc,\Cc,\Cc_R)^{\co},\Dd) \iso \bFun_{R\radj}(\Cc\catop,\Dd).
		\]
	\end{enumerate}
\end{corollary}
\begin{proof}
	This is immediate from the equivalence $\bSpan(\Cc,\Cc,\Cc_R)\catop \simeq \bSpan(\Cc,\Cc_R,\Cc)$ and the following three equivalences:
	\begin{align*}
		\bFun_{L\ladj}(\Cc,\Dd) &\simeq \bFun_{L\radj}(\Cc,\Dd^{\co})^{\co}; \\
		\bFun_{R\ladj}(\Cc\catop,\Dd) &\simeq \bFun_{R\radj}(\Cc,\Dd^{\op})\catop; \\
		\bFun_{R\radj}(\Cc\catop,\Dd) &\simeq \bFun_{R\ladj}(\Cc,\Dd^{\coop})^{\coop}. \qedhere
	\end{align*}
\end{proof}

\section{A 2-categorical unfurling construction}
Barwick's unfurling construction \cite{Barwick2017SpectralMackey} provides a method for extending a right $L$-adjointable functor $F\colon \Cc \to \Cat_{\infty}$ to a functor $\Unf(F)\colon \Span(\Cc,\Cc_L,\Cc) \to \Cat_{\infty}$: if $p\colon\Ee\to\Cc$ is the cocartesian fibration classified by $F$, Barwick proved that the induced functor $\Span(p)\colon \Span(\Ee,{p^{-1}(\Cc_L) \cap \Ee^{p\textup{-cart}}},\Ee) \to \Span(\Cc,\Cc_L,\Cc)$ is again a cocartesian fibration whose straightening $\Span(\Cc,\Cc_L,\Cc) \to \Cat_{\infty}$ extends $F$. In this section we will show that this construction can be enhanced to even produce a 2-functor $\bSpan(\Cc,\Cc_L,\Cc) \to \bCat_{\infty}$, {which must then necessarily agree} with the universal extension obtained from the universal property of $\bSpan(\Cc,\Cc_L,\Cc)$.

\subsection{Cocartesian fibrations of 2-categories}
\label{subsec:2-categorical-unstraightening}

We start with a brief recollection on cocartesian fibrations between $(\infty,2)$-categories, see for instance \cite[Section~5]{Nuiten2021Straightening} or \cite[Chapter~11]{GaitsgoryRozenbluym2017StudyDAG}.

\begin{definition}[2-cocartesian fibrations]
	Let $p\colon \Ee \to \Cc$ be a 2-functor between $(\infty,2)$-categories. We say that $p$ is a \textit{2-cocartesian fibration} if it satisfies the following conditions:
	\begin{enumerate}[(1)]
		\item (Homwise cartesian fibration) For all objects $\tilde{X},\tilde{Y} \in \Ee$, the induced map on Hom-categories
		\[
		\Hom_{\Ee}(\tilde{X},\tilde{Y}) \to \Hom_{\Cc}(p\tilde{X},p\tilde{Y})
		\]
		is a cartesian fibration of $\infty$-categories, and for morphisms $\tilde{f}\colon \tilde{X}' \to \tilde{X}$ and $\tilde{g}\colon \tilde{Y} \to \tilde{Y}'$ in $\Ee$, the composition functor
		\[
		\tilde{g} \circ (-) \circ \tilde{f}\colon \Hom_{\Ee}(\tilde{X},\tilde{Y}) \to \Hom_{\Ee}(\tilde{X}',\tilde{Y}')
		\]
		preserves cartesian 2-morphisms.

		\item (Enough 2-cocartesian lifts) For every morphism $f\colon X \to Y$ in $\Cc$ and every lift $\tilde{X} \in \Ee$ of $X$, there exists a 2-cocartesian lift $\tilde{f}\colon \tilde{X} \to \tilde{Y}$ of $f$, meaning that for every object $\tilde{Z} \in \Ee$, the square
		\[
		\begin{tikzcd}
			\Hom_{\Ee}(\tilde{Y},\tilde{Z}) \rar{- \circ \tilde{f}} \dar[swap]{p} & \Hom_{\Ee}(\tilde{X},\tilde{Z}) \dar{p} \\
			\Hom_{\Cc}(p\tilde{Y},p\tilde{Z}) \rar{- \circ p\tilde{f}} & \Hom_{\Cc}(p\tilde{X},p\tilde{Z})
		\end{tikzcd}
		\]
		is a pullback square.
	\end{enumerate}

	We say that $p$ is a \textit{1-cocartesian fibration} if it is a 2-cocartesian fibration and the functors $\Hom_{\Ee}(\tilde{X},\tilde{Y}) \to \Hom_{\Cc}(\tilde{X},Y)$ in (1) are in fact right fibrations.

    Dually, we have notions of \emph{2-cartesian fibrations} and \emph{1-cartesian fibrations} (which are homwise \emph{left} fibrations).

	Given an $(\infty,2)$-category $\Cc$, we let $\bCocart_2(\Cc)$ denote the locally full sub-2-category of $(\bCat_{(\infty,2)})_{/\Cc}$ spanned by the 2-cocartesian fibrations $\Ee \to \Cc$ and those 2-functors which preserve 2-cocartesian 1-morphisms and cartesian 2-morphisms. We let $\bCocart_1(\Cc) \subseteq \bCocart_2(\Cc)$ denote the full subcategory spanned by the 1-cocartesian fibrations.
\end{definition}

\begin{theorem}[Straightening and unstraightening for $(\infty,2)$-categories, {\cite[Theorem~6.1, Remark~6.12, Theorem~6.20]{Nuiten2021Straightening}}]
	For every $(\infty,2)$-category $\Cc$, there is an equivalence of $(\infty,2)$-categories
	\[
	\Str^{\cc}\colon \bCocart_2(\Cc) \iso \bFun_2(\Cc,\bCat_{(\infty,2)})
	\]
	which is natural in $\Cc$. Furthermore, this equivalence restricts to an equivalence
	\[
	\Str^{\cc}\colon \bCocart_1(\Cc) \iso \bFun_2(\Cc,\bCat_{(\infty,1)}),
	\]
	and if $\Cc$ is an $(\infty,1)$-category, this recovers the usual straightening equivalence. Similarly, there is an equivalence
	\[
	\Str^{\ct}\colon \bCart_2(\Cc) \iso \bFun_2(\Cc\catop,\bCat_{(\infty,2)})
	\]
    which restricts to $\bCart_1(\Cc) \iso \bFun_2(\Cc\catop,\bCat_{\infty})$. \qed
\end{theorem}

There is the following useful recognition principle for 2-cocartesian fibrations:

\begin{lemma}
	Let $p\colon\Ee\to\Cc$ be a homwise right fibration. Then a morphism $f\colon x \to y$ in $\Ee$ is $2$-cocartesian if and only if it is cocartesian for the induced map on underlying $(\infty,1)$-categories.
	\begin{proof}
		For a $2$-cocartesian morphism, we require the left square in the following diagram to be a pullback, while being cocartesian for the underlying $(\infty,1)$-categories amounts to the induced square on groupoid cores (depicted on the right) being a pullback for every $z\in\Ee$:
		\[\hskip-14.14pt\hfuzz=14.15pt
		\begin{tikzcd}[cramped]
			\Hom_{\Ee}(y,z)\arrow[r, "-\circ f"]\arrow[d,"p"'] & \Hom_{\Ee}(x,z)\arrow[d,"p"]\\
			\Hom_{\Cc}(p(y),p(z))\arrow[r, "-\circ p(f)"] & \Hom_{\Cc}(p(x),p(z))
		\end{tikzcd}
		\qquad
		\begin{tikzcd}[cramped]
			\iota\Hom_{\Ee}(y,z)\arrow[r, "-\circ f"]\arrow[d,"p"'] & \iota\Hom_{\Ee}(x,z)\arrow[d,"p"]\\
			\iota\Hom_{\Cc}(p(y),p(z))\arrow[r, "-\circ p(f)"] & \iota\Hom_{\Cc}(p(x),p(z))
		\end{tikzcd}
		\]
		As $\iota$ is a right adjoint, we immediately see that if $f$ is $2$-cocartesian, then it is cocartesian for the underlying functor of $(\infty,1)$-categories. Conversely, note that by assumption on $p$ both vertical maps in the square on the left are right fibrations. Thus, the comparison map $\Hom(y,z)\to\Hom(x,z)\times_{\Hom(p(x),p(z))}\Hom(p(y),p(z))$ is a map of right fibrations over $\Hom(p(y),p(z))$ and so we can check whether it is an equivalence after passing to fibers. Using that fibers of a right fibration are $\infty$-groupoids, these can be identified with the fibers of the analogous comparison map $\iota\Hom(y,z)\to\iota\Hom(x,z)\times_{\iota\Hom(p(x),p(z))}\iota\Hom(p(y),p(z))$, and the claim follows.
	\end{proof}
\end{lemma}

\begin{corollary}
	\label{cor:Check_Cocartesian_2Functor_Underlying}
	Let $p\colon \Ee \to \Cc$ be a homwise right fibration. Then it is a 1-cocartesian fibration if and only if its underlying functor of $(\infty,1)$-categories is a cocartesian fibration.\qed
\end{corollary}

\subsection{Cocartesian fibrations of span 2-categories}
We will now prove a general result which produces 1-cocartesian fibrations between span 2-categories, refining
the analogous $(\infty,1)$-categorical result by Barwick \cite{Barwick2017SpectralMackey}.

\begin{definition}[{cf.\ \cite[Definition~4.1.3, Remark~4.1.5]{hopkinsLurie2013ambidexterity}}]
	\label{def:BC_Fibration}
	Let $(\Cc,\Cc_L,\Cc_R)$ be an adequate triple. A functor $p\colon \Ee \to \Cc$ is called a \textit{Beck--Chevalley fibration} with respect to $(\Cc_L,\Cc_R)$ if the following three conditions are satisfied:
	\begin{enumerate}[(1)]
		\item For every morphism $l\colon X \to Y$ in $\Cc_L$ and every object $\tilde{Y}$ of $\Ee$ lifting $Y$, there exists a $p$-cartesian lift $\tilde{l}\colon \tilde{X} \to \tilde{Y}$ of $l$;
		\item For every morphism $r\colon X \to Y$ in $\Cc_R$ and every object $\tilde{X}$ of $\Ee$ lifting $X$, there exists a $p$-cocartesian lift $\tilde{r}\colon \tilde{X} \to \tilde{Y}$ of $r$;
		\item Consider a commutative square
		\[
		\begin{tikzcd}
			\tilde{X}' \arrow[r, "\tilde{r'}"] \arrow[d, "\tilde{l'}"'] & \tilde{X} \arrow[d, "\tilde{l}"] \\
			\tilde{Y}' \arrow[r, "\tilde{r}"] & \tilde{Y}
		\end{tikzcd}
		\]
		in $\Ee$ whose image in $\Cc$ is a pullback square satisfying $p(\tilde{l}) \in \Cc_L$ and $p(\tilde{r'}) \in \Cc_R$, and for which $\tilde{r}$ is $p$-cocartesian while $\tilde{l'}$ is $p$-cartesian. Then the morphism $\tilde{r'}$ is $p$-cocartesian if and only if $\tilde{l}$ is $p$-cartesian.
	\end{enumerate}
	We let $\mathrm{Fib}^{\BC}_{\Cc_L,\Cc_R}(\Cc) \subseteq \Cat_{/\Cc}$ denote the (non-full) subcategory spanned by the Beck--Chevalley fibrations $p\colon \Ee \to \Cc$ with respect to $(\Cc_L,\Cc_R)$ and by those functors $\Ee \to \Ee'$ over $\Cc$ which preserve $p$-cartesian morphisms over $\Cc_L$ and $p$-cocartesian morphisms over $\Cc_R$.
\end{definition}

\begin{notation}
\label{nota:Adequate_Triple_Unfurling}
        In the situation of the previous definition, we obtain the following two wide subcategories of $\Ee$:
	\[
		\Ee_R \coloneqq  p^{-1}(\Cc_R) \qquadtext{ and } \Ee_L^{p\textup{-cart}} \coloneqq  p^{-1}(\Cc_L) \cap \Ee^{p\textup{-cart}},
	\]
	where $\Ee^{p\textup{-cart}}$ denotes the wide subcategories of $\Ee$ spanned by the $p$-cartesian morphisms. By \cite[Proposition~2.6]{HHLN2022TwoVariable}, the triple $(\Ee,\Ee_L^{p\textup{-cart}},\Ee_R)$ is adequate, and the functor $p$ defines a morphism of adequate triples to $(\Cc,\Cc_L,\Cc_R)$.
\end{notation}

\begin{observation}
	\label{obs:Beck_Chevalley_Cancellation}
	Let $p\colon \Ee \to \Cc$ be a Beck--Chevalley fibration with respect to $(\Cc_L,\Cc_R)$, and let $\tilde{r}$ be a $p$-cocartesian lift of some morphism $r\colon X \to Y$ in $\Cc_R$. Then $\tilde{r}$ is also cocartesian with respect to the functor $p\vert_{\Ee_R}\colon \Ee_R \to \Cc_R$. Similarly, every $p$-cartesian lift $\tilde{l}$ of a morphism $l$ in $\Cc_L$ is also cartesian with respect to the functor $p\vert_{\Ee_L^{p\textup{-cart}}}\colon \Ee_L^{p\textup{-cart}} \to \Cc_L$, using the left cancelation for $p$-cartesian morphisms.
\end{observation}

\begin{theorem}
	\label{thm:Cocartesian_Fibrations_Of_Span_2Categories}
	Let $(\Cc,\Cc_L,\Cc_R)$ be an adequate triple and let $p\colon \Ee \to \Cc$ be a Beck--Chevalley fibration with respect to $(\Cc_L,\Cc_R)$. Assume that $p$ is a cartesian fibration, or that the morphisms in $\Cc_L$ are left-cancelable.\footnote{That is, given morphisms $f\colon X \to Y$ and $g\colon Y \to Z$ in $\Cc$, if both $g$ and $gf$ are in $\Cc_L$ then so is $f$.} Then the induced 2-functor
	\begin{equation}\label{eq:bspan-p}
		\bSpan(p)\colon \bSpan(\Ee,\Ee_L^{p\textup{-cart}},\Ee_R) \to \bSpan(\Cc,\Cc_L,\Cc_R)
	\end{equation}
	is a 1-cocartesian fibration.
\end{theorem}

For the proof we will use:

\begin{lemma}\label{lemma:homwise-rfib}
    Let $(\Cc,\Cc_L,\Cc_R)$ be an adequate triple and let $p\colon\Ee\to\Cc$ be any functor. Assume at least one of the following conditions is satisfied:
    \begin{enumerate}[(1)]
        \item $p$ is a cartesian fibration;
        \item $\Cc_L$ is left-cancelable, and $p$ admits cartesian lifts of all morphisms in $\Cc_L$.
    \end{enumerate}
    Then the 2-functor $\bSpan(p)$ in $(\ref{eq:bspan-p})$ is a homwise right fibration.
\end{lemma}
\begin{proof}
     We have to show that for all objects $\tilde{X}$ and $\tilde{Y}$ of $\Ee$ the induced functor
	\[
		\Hom_{\bSpan(\Ee,\Ee_L^{p\textup{-cart}},\Ee_R)}(\tilde{X},\tilde{Y}) \xrightarrow{\bSpan(p)} \Hom_{\bSpan(\Cc,\Cc_L,\Cc_R)}(p\tilde{X},p\tilde{Y})
	\]
	is a right fibration. {By definition, the source of this functor is the full subcategory of $\Ee_{/\tilde{X}} \times_{\Ee} \Ee_{/\tilde{Y}}$ spanned by those objects $\tilde{X} \xleftarrow{\tilde{l}} \tilde{Z} \xrightarrow{\tilde{r}} \tilde{Y}$ with $\tilde{l} \in \Ee_L^{p\textup{-cart}}$ and $\tilde{r} \in \Ee_R$; we may similarly identify its target as a full subcategory of $\Cc_{/p\tilde{X}} \times_{\Cc} \Cc_{/p\tilde{Y}}$.} First observe that every morphism in the target admits \textit{some} lift to the source with specified codomain. Indeed, consider an object $\tilde{X} \xleftarrow{\tilde{l}} \tilde{Z} \xrightarrow{\tilde{r}} \tilde{Y}$ in the source, and consider a morphism in the target of the form
	\[
	\begin{tikzcd}[row sep=tiny, column sep=small]
		& Z' \dlar[swap]{l'} \drar{r'} \ar{dd}{f} \\
		p\tilde{X} && p\tilde{Y}. \\
		& p\tilde{Z} \ular{p\tilde{l}} \urar[swap]{p\tilde{r}}
	\end{tikzcd}
	\]
	Either of the two assumptions of the lemma imply that there exists a $p$-cartesian lift $\tilde{f}\colon \tilde{Z'} \to \tilde{Z}$: either because $p$ is a cartesian fibration, or because $f$ lies in $\Cc_L$ due to the left cancelation property. It then follows that the morphism $\tilde{l'} \coloneqq  \tilde{l} \circ \tilde{f}$ lies in $\Ee_L^{p\textup{-cart}}$ and that the morphism $\tilde{r'} \coloneqq  \tilde{r} \circ \tilde{f}$ lies in $\Ee_R$, and hence the resulting diagram
	\[
	\begin{tikzcd}[row sep=tiny, column sep=small]
		& \tilde{Z'} \dlar[swap]{\tilde{l'}} \drar{\tilde{r'}} \ar{dd}{\tilde{f}} \\
		\tilde{X} && \tilde{Y} \\
		& \tilde{Z}, \ular{\tilde{l}} \urar[swap]{\tilde{r}}
	\end{tikzcd}
	\]
	defines a morphism in $\Hom_{\bSpan(\Ee,\Ee_L^{p\textup{-cart}},\Ee_R)}(\tilde{X},\tilde{Y})$ which lifts the given morphism. It will now suffice to show that every morphism in $\Hom_{\bSpan(\Ee,\Ee_L^{p\textup{-cart}},\Ee_R)}(\tilde{X},\tilde{Y})$ is cartesian over $\Hom_{\bSpan(\Cc,\Cc_L,\Cc_R)}(p\tilde{X},p\tilde{Y})$. But this is an immediate consequence of the fact that for every such morphism
	\[
	\begin{tikzcd}[row sep=tiny, column sep=small]
		& \tilde{Z'} \dlar[swap]{\tilde{l'}} \drar{\tilde{r'}} \ar{dd}{\tilde{f}} \\
		\tilde{X} && \tilde{Y} \\
		& \tilde{Z}, \ular{\tilde{l}} \urar[swap]{\tilde{r}}
	\end{tikzcd}
	\]
	the morphism $\tilde{f}\colon \tilde{Z'} \to \tilde{Z}$ is $p$-cartesian by left-cancelation of $p$-cartesian morphisms. This finishes the proof.
\end{proof}

\begin{proof}[Proof of Theorem~\ref{thm:Cocartesian_Fibrations_Of_Span_2Categories}]
	We first show
    that the underlying functor of $(\infty,1)$-categories $\Span(p)\colon \Span(\Ee,\Ee_L^{p\textup{-cart}},\Ee_R) \to \Span(\Cc,\Cc_L,\Cc_R)$ is a cocartesian fibration. Indeed, given a morphism
	\[
		X \xleftarrow{l} Z \xrightarrow{r} Y
	\]
	in $\Span(\Cc,\Cc_L,\Cc_R)$ and a lift $\tilde{X}$ of $X$, we may lift it to a morphism
	\[
		\tilde{X} \xleftarrow{\tilde{l}} \tilde{Z} \xrightarrow{\tilde{r}} \tilde{Y}
	\]
	in $\Span(\Ee,\Ee_L^{p\textup{-cart}},\Ee_R)$, where $\tilde{l}$ is a $p$-cartesian lift of $l$ and where $\tilde{r}$ is a $p$-cocartesian lift of $r$. This morphism is then a $\Span(p)$-cocartesian lift by applying \cite[Theorem~3.1]{HHLN2022TwoVariable} to the functor $p\colon (\Ee,\Ee_L^{p\textup{-cart}},\Ee_R) \to (\Cc,\Cc_L,\Cc_R)$. The conditions of that theorem are satisfied due to $p$ being a Beck--Chevalley fibration, and using \Cref{obs:Beck_Chevalley_Cancellation}.

	By \Cref{cor:Check_Cocartesian_2Functor_Underlying}, it then only remains to show that the 2-functor $\bSpan(p)$ is also a homwise right fibration, which is an instance of the previous lemma.
\end{proof}

\begin{construction}
    We now enhance the above construction to a functor
    \begin{equation}\label{eq:bspan}
        \bSpan\colon \mathrm{Fib}^{\BC}_{\Cc_L,\Cc_R}(\Cc) \to \Cocart_1(\bSpan(\Cc,\Cc_L,\Cc_R)).
    \end{equation}
    Since a Beck--Chevalley fibration $p\colon \Ee \to \Cc$ becomes a morphism of adequate triples once equipping its source with the subcategories $\Ee_L^{p\textup{-cart}}$ and $\Ee_R$, we have an evident functor from $\mathrm{Fib}^{\BC}_{\Cc_L,\Cc_R}(\Cc)$ to $\AdTrip_{/(\Cc,\Cc_L,\Cc_R)}$. Composing this with the functor $\bSpan\colon \AdTrip \to \Cat_{(\infty,2)}$, this yields a functor
    \[
    \bSpan\colon \mathrm{Fib}^{\BC}_{\Cc_L,\Cc_R}(\Cc) \to (\Cat_{(\infty,2)})_{/\bSpan(\Cc,\Cc_L,\Cc_R)}.
    \]
    By \Cref{thm:Cocartesian_Fibrations_Of_Span_2Categories} this functor on objects lands in $\Cocart_1(\bSpan(\Cc,\Cc_L,\Cc_R))$. From the explicit description of the $\Span(p)$-cocartesian morphisms given in \Cref{thm:Cocartesian_Fibrations_Of_Span_2Categories}, one similarly observes that for every morphism $\Ee \to \Ee'$ of Beck--Chevalley fibrations over $(\Cc,\Cc_L,\Cc_R)$ the resulting 2-functor $\bSpan(\Ee,\Ee_L^{p\textup{-cart}},\Ee_R) \to \bSpan(\Ee',{\Ee'}_L^{p\textup{-cart}},\Ee'_R)$ is 1-cocartesian over $\bSpan(\Cc,\Cc_L,\Cc_R)$, i.e.~$\bSpan$ restricts to $(\ref{eq:bspan})$.
\end{construction}

\subsection{Universality of the unfurling construction}

We will now specialize the results from the previous subsection to {Beck--Chevalley fibrations obtained} via either cartesian or cocartesian unstraightening.

\begin{proposition}[{cf.\ \cite[Proposition~3.4.4 and 3.4.10]{MacPherson2022Bivariant}}]
\label{prop:BC_Fibration_From_Adjointable}
	Let $(\Cc,\Cc_L)$ and $(\Cc,\Cc_R)$ be span pairs.
	\begin{enumerate}[(1)]
		\item The cocartesian unstraightening equivalence $\Un^{\cc}\colon \Fun(\Cc,\Cat_{\infty}) \iso \Cocart(\Cc)$ restricts to an equivalence
		\[
		\Un^{\cc}\colon \Fun_{L\radj}(\Cc,\Cat_{\infty}) \iso \mathrm{Fib}^{\BC}_{(\Cc_L,\Cc)}(\Cc).
		\]
		\item The cartesian unstraightening equivalence $\Un^{\ct}\colon \Fun(\Cc\catop,\Cat_{\infty}) \iso \bCart(\Cc)$ restricts to an equivalence
		\[
		\Un^{\ct}\colon \Fun_{R\ladj}(\Cc\catop,\Cat_{\infty}) \iso \mathrm{Fib}^{\BC}_{(\Cc,\Cc_R)}(\Cc).
		\]
	\end{enumerate}
\end{proposition}
\begin{proof}
	We prove part (1); part (2) is entirely dual. Let $F\colon \Cc \to \Cat_{\infty}$ be a functor and let $p\colon \Ee \to \Cc$ denote its cocartesian unstraightening. We start by showing that $F$ is covariantly right $L$-adjointable if and only if $p$ is a Beck--Chevalley fibration with respect to $(\Cc,\Cc_L,\Cc)$. Since $p$ is a cocartesian fibration, condition (2) in \Cref{def:BC_Fibration} is automatic. Condition (1) in \Cref{def:BC_Fibration} corresponds to condition (1) in \Cref{def:BC_Functor}. In this case the square in condition (3) takes the form
	\[
	\begin{tikzcd}
		h^*\tilde{X} \arrow[r, "\tilde{h}"] \arrow[d, "\tilde{k}"'] & \tilde{X} \arrow[d, "\tilde{f}"] \\
		\tilde{Y}' \arrow[r, "\tilde{g}"] & f_!\tilde{X}.
	\end{tikzcd}
	\]
	The map $\tilde{k}$ induces a map $k_!h^*\tilde{X} \to \tilde{Y}'$ which is an equivalence if and only if $\tilde{k}$ is $p$-cocartesian, and similarly $\tilde{g}$ induces a map $\tilde{Y}' \to g^*f_!\tilde{X}$ which is an equivalence if and only if $\tilde{g}$ is $p$-cartesian. Since the composite of these two maps is the Beck--Chevalley map $k_!h^*\tilde{X} \to g^*f_!\tilde{X}$, it follows by 2-out-of-3 that condition (3) from \Cref{def:BC_Fibration} is equivalent to condition (2) from \Cref{def:BC_Functor}. We conclude that $F$ is covariantly right $L$-adjointable if and only if $p$ is a Beck--Chevalley fibration.

	To finish the proof, we now observe that by definition a natural transformation $F \Rightarrow G$ between two right $L$-adjointable functors $\Cc \to \Cat_{\infty}$ is right $L$-adjointable if and only if the map over $\Cc$ between their cocartesian unstraightenings (which always preserves cocartesian morphisms) preserves cartesian morphisms over $\Cc_L$, which is equivalent to being a morphism of Beck--Chevalley fibrations.
\end{proof}

We are finally ready to introduce the 2-categorical enhancement of Barwick's unfurling construction:

\begin{definition}[Unfurling]\label{def:unfurling}
	Let $(\Cc,\Cc_L)$ and $(\Cc,\Cc_R)$ be span pairs.
	\begin{enumerate}[(1)]
		\item We define the \textit{covariant unfurling construction} as the following composite:
		\begin{align*}
			\Unf^{\cov}\colon \Fun_{L\radj}(\Cc,\Cat_{\infty}) &\xrightarrow[\raisebox{2.5pt}{$\scriptstyle\sim$}]{\Un^{\cc}} \mathrm{Fib}^{\BC}_{(\Cc_L,\Cc)}(\Cc) \\
			&\xrightarrow{\bSpan} \Cocart(\bSpan(\Cc,\Cc_L,\Cc)) \\
			&\xrightarrow[\raisebox{2.5pt}{$\scriptstyle\sim$}]{\Str^{\cc}} \Fun_2(\bSpan(\Cc,\Cc_L,\Cc),\bCat_{\infty}).
		\end{align*}
		\item We define the \textit{contravariant unfurling construction} as the following composite:
		\begin{align*}
			\Unf^{\contra}\colon \Fun_{R\ladj}(\Cc\catop,\Cat_{\infty}) &\xrightarrow[\raisebox{2.5pt}{$\scriptstyle\sim$}]{\Un^{\ct}} \mathrm{Fib}^{\BC}_{(\Cc,\Cc_R)}(\Cc) \\
			&\xrightarrow{\bSpan} \Cocart(\bSpan(\Cc,\Cc,\Cc_R)) \\
			&\xrightarrow[\raisebox{2.5pt}{$\scriptstyle\sim$}]{\Str^{\cc}} \Fun_2(\bSpan(\Cc,\Cc,\Cc_R),\bCat_{\infty}).
		\end{align*}
	\end{enumerate}
\end{definition}

\begin{theorem}[Universality of unfurling]
\label{thm:Universality_Unfurling}
	Let $(\Cc,\Cc_L)$ and $(\Cc,\Cc_R)$ be span pairs.
	\begin{enumerate}[(1)]
		\item The covariant unfurling construction provides an inverse to the restriction functor
		\[
			\Fun_2(\bSpan(\Cc,\Cc_L,\Cc),\bCat_{\infty}) \iso \Fun_{L\radj}(\Cc,\Cat_{\infty}).
		\]
		\item The contravariant unfurling construction provides an inverse to the restriction functor
		\[
			\Fun_2(\bSpan(\Cc,\Cc,\Cc_R),\bCat_{\infty}) \iso \Fun_{R\ladj}(\Cc\catop,\Cat_{\infty}).
		\]
	\end{enumerate}
\end{theorem}
\begin{proof}
	By the universal properties of the inclusions $\Cc \hookrightarrow \bSpan(\Cc,\Cc_L,\Cc)$ and $\Cc\catop \hookrightarrow \bSpan(\Cc,\Cc,\Cc_R)$ we know that these restriction functors are equivalences, hence it will suffice to show that the two unfurling constructions define sections.

	In the covariant case, this amounts to showing that for a Beck--Chevalley fibration $p\colon \Ee \to \Cc$ with respect to $(\Cc_L,\Cc)$, the base change of the 1-cocartesian fibration $\bSpan(p)\colon \bSpan(\Ee,\Ee_L^{p\textup{-cart}},\Ee) \to \bSpan(\Cc,\Cc_L,\Cc)$ along the inclusion $\Cc \simeq \bSpan(\Cc,\Cc^{\simeq},\Cc)\hookrightarrow \bSpan(\Cc,\Cc_L,\Cc)$ is the cocartesian fibration $p\colon \Ee \to \Cc$. But this is clear: this base change is given by the functor
	\[
		\bSpan(p)\colon \bSpan(\Ee,\Ee_L^{p\textup{-cart}} \cap p^{-1}(\Cc^{\simeq}),\Ee) \to \bSpan(\Cc,\Cc^{\simeq},\Cc) \simeq \Cc,
	\]
	and the source is naturally equivalent to $\Ee$ since every $p$-cartesian morphism in $\Ee$ whose image in $\Cc$ is an equivalence is already an equivalence.

	In the contravariant case, we start with a cartesian fibration $p\colon \Ee \to \Cc$ which is the cartesian unstraightening of some functor $F\colon \Cc\catop \to \Cat_{\infty}$. The base change of the resulting 1-cocartesian fibration $\bSpan(p)\colon \bSpan(\Ee,\Ee^{p\textup{-}\cart},\Ee_R) \to \Span(\Cc,\Cc,\Cc_R)$ along the inclusion $\Cc\catop \simeq \bSpan(\Cc,\Cc,\Cc^{\simeq}) \hookrightarrow \bSpan(\Cc,\Cc,\Cc_R)$ is given by the cocartesian fibration
	\begin{equation}\label{eq:this-cocartesian-fibration}
		\Span(p)\colon \Span(\Ee,\Ee^{p\textup{-}\cart},\Ee^{\mathrm{fw}}) \to \Span(\Cc,\Cc,\Cc^{\simeq}) \simeq \Cc\catop,
	\end{equation}
	where $\Ee^{\mathrm{fw}} \subseteq \Ee$ denotes the wide subcategory spanned by those morphisms in $\Ee$ whose image in $\Cc$ is invertible. By the main result of \cite{BGN2018Dualizing}, the cocartesian unstraightening $\Cc\catop \to \Cat_{\infty}$ of $(\ref{eq:this-cocartesian-fibration})$ is naturally equivalent to $F$.
	This finishes the proof.
\end{proof}

\section{Yet another unfurling construction}
So far, we have discussed how to make the $\bSpan$-extensions of contravariantly left $R$-adjointable functors and covariantly right $L$-adjointable functors explicit. In this section, we will explain how to similarly get a concrete grip on the extension of a contravariantly right $R$-adjointable functor, and we apply this to compare the two competing constructions of `cartesian normed structures,' due to Nardin--Shah and Cnossen--Haugseng--Lenz--Linskens.

\subsection{Adding right adjoints to a contravariant functor}
Recall from Corollary~\ref{cor:other-univ-properties} that $\Cc^{\op}\hookrightarrow\bSpan(\Cc,\Cc,\Cc_R)^{\co}$ is the initial contravariant right $R$-adjointable functor. The following proposition shows that this universal property is again implemented by a suitable unfurling construction:

\begin{proposition}\label{prop:subtle}
    Let $\phi\colon\Cc^{\op}\to\Cat_\infty$ be right $R$-adjointable and let $p\colon\Ee\to\Cc$ be the cartesian fibration for $\phi$. Write $\Ee_R^{p\textup{-cart}}\subset\Ee$ for the subcategory of cartesian lifts of maps in $\Cc_R$. Then $(\Ee,\Ee_R^{p\textup{-cart}})$ is a span pair, $(\Ee,\Ee_R^{p\textup{-cart}})\to(\Cc,\Cc_R)$ is a map of span pairs, and {the 2-functor}
    \begin{equation}\label{eq:subtle-fib}
        \bSpan(p)^{\co}\colon \bSpan(\Ee,\Ee_R^{p\textup{-cart}},\Ee)^{\co}\to\bSpan(\Cc,\Cc_R,\Cc)^{\co}
    \end{equation}
    is a $1$-cartesian fibration, whose cartesian straightening is the unique extension of $\phi$ to $\bSpan(\Cc,\Cc_R,\Cc)^{\coop} = \bSpan(\Cc,\Cc,\Cc_R)^{\co}$.
\end{proposition}

The proof will rely on a cheap trick and the following lemma:

\begin{lemma}
	Let $\phi\colon\Cc^{\op}\to\Cat_\infty$ be right $R$-adjointable. Then there exists a pointwise fully faithful {right $R$-adjointable transformation} $\phi\hookrightarrow\hat\phi$ of right $R$-adjointable functors such that for every $f\colon A\to B$ in $\Cc$ the functor $f^*\colon\hat\phi(B)\to\hat\phi(A)$ admits a \emph{left} adjoint $f_!$.
	\begin{proof}
		We will instead show how a \emph{left} $R$-adjointable functor $\phi$ embeds into a $\hat\phi$ such that all restrictions have \emph{right} adjoints. The actual lemma will follow by applying this to $(-)^{\op}\circ\phi$ and then {passing to opposite categories again}.

		We define $\hat\phi(A)\coloneqq\PSh(\phi(A))$, with functoriality induced from $\phi$ via left Kan extension. The Yoneda embeddings then define the required fully faithful embedding $\phi\hookrightarrow\hat\phi$ by \cite[Theorem~8.1]{HHLN2022TwoVariable} or \cite[Theorem~2.4]{Ramzi-Yoneda}. The existence of right adjoints is then clear, while $R$-adjointability follows from $2$-functoriality of $\PSh(-)$.
	\end{proof}
\end{lemma}

\begin{proof}[Proof of Proposition~\ref{prop:subtle}]
    Note that $(\Ee,\Ee_R^{p\textup{-cart}},\Ee)\to(\Cc,\Cc_R,\Cc)$ is a map of adequate triples by \cite[Proposition~2.6]{HHLN2022TwoVariable}. The only non-trivial statement that remains is then that $(\ref{eq:subtle-fib})$ is indeed a $1$-cartesian fibration.

    By Lemma~\ref{lemma:homwise-rfib}${}^{\co}$, $(\ref{eq:subtle-fib})$ induces \emph{left} fibrations on hom categories, so {by \Cref{cor:Check_Cocartesian_2Functor_Underlying}${}^{\coop}$} we only have to show that the underlying functor $\Span(\Ee,\Ee_R^{p\textup{-cart}},\Ee)\to\Span(\Cc,\Cc_R,\Cc)$ of $(\infty,1)$-categories is a cartesian fibration.

    For this, we first consider the special case that all restricitions $r^*\colon\phi(B)\to\phi(A)$ also admit left adjoints (no Beck--Chevalley condition required). In this case $\Ee\to\Cc$ is also a \emph{cocartesian} fibration, and as such it is classified by the functor $\Cc\to\Cat_\infty$ obtained from $\phi$ by passing to left adjoints. Observe that this is still right $R$-adjointable, since the total mate of a Beck--Chevalley map in a commutative square is the Beck--Chevalley map of the corresponding adjoints. We thus see that $p$ is a Beck--Chevalley fibration with respect to $(\Cc_R,\Cc)$ by \Cref{prop:BC_Fibration_From_Adjointable}, hence it follows from Theorem~\ref{thm:Cocartesian_Fibrations_Of_Span_2Categories} that $\Span(p)\colon\Span(\Ee,\Ee_R^{p\textup{-cart}},\Ee)\to\Span(\Cc,\Cc_R,\Cc)$ is a \emph{cocartesian} fibration, and the functor classifying it sends a span
    \[
        X \xleftarrow{\;f\;} Y\xrightarrow{\;g\;} Z
    \]
    to $g_!f^*$. As this is a left adjoint (as a composition of left adjoints), we see that $\Span(p)$ is also cartesian, as claimed.

    In general, we use the lemma to obtain an inclusion $\phi\hookrightarrow\hat\phi$ such that all restrictions for $\hat\phi$ have left adjoints. If we write $q\colon\hat\Ee\to\Cc$ for the cartesian unstraightening of $\hat\phi$, then we have an inclusion $\Ee\hookrightarrow\hat\Ee$ that preserves cartesian edges, i.e.~it restricts to $\Ee_R^{p\textup{-cart}}\to\hat\Ee_R^{p\textup{-cart}}$. Note that this is a map of adequate triples as the proof of \cite[Proposition~2.6]{HHLN2022TwoVariable} shows that a square
    \[
        \begin{tikzcd}
            A'\arrow[r, "f'"]\arrow[d] & B'\arrow[d]\\
            A\arrow[r, "f"'] & B
        \end{tikzcd}
    \]
    in $\Ee$ with $f\in\Ee_R^{p\textup{-cart}}$ is a pullback in $\Ee$ or $\hat\Ee$ if and only if its image in $\Cc$ is a pullback and $f'\in\Ee_R^{p\textup{-cart}}$.

    We may therefore view $\Span(\Ee,\Ee_R^{p\textup{-cart}},\Ee)$ as a subcategory of $\Span(\hat\Ee,\hat\Ee_R^{p\textup{-cart}},\hat\Ee)$. Note that this is in fact a full subcategory, as $\Ee$ contains any edge of $\hat\Ee_R^{p\textup{-cart}}$ with target in $\Ee$. To complete the proof it therefore suffices to show that also $\Span(\Ee,\Ee_R^{p\textup{-cart}},\Ee)\subset\Span(\hat\Ee,\hat\Ee_R^{p\textup{-cart}},\hat\Ee)$ is closed under cartesian pullback. This may be checked separately over $\Cc$ and $\Cc_R^{\op}$, where this translates to naturality of $\phi\hookrightarrow\hat\phi$ and the Beck--Chevalley condition, respectively.
\end{proof}

\subsection{Cartesian normed structures} As promised, we can now use the uniqueness part of the universal property to compare some competing constructions in parametrized higher algebra. We begin with the following observation:

\begin{construction}\label{constr:nardin-shah-unfurl}
    Let $\phi\colon\Cc^{\op}\to\Cat_\infty$ be right $R$-adjointable. Then the composite $(-)^{\op}\circ\phi\colon\Cc^{\op}\to\Cat_{\infty}$ is \emph{left} $R$-adjointable, and therefore we can extend it to $\bSpan(\Cc,\Cc,\Cc_R)$ via the contravariant unfurling construction  (Definition~\ref{def:unfurling}). Postcomposing with $(-)^{\op}\colon\bCat_\infty\to\bCat_\infty^{\co}$, we obtain a $2$-functor $\bSpan(\Cc,\Cc,\Cc_R)\to\bCat^{\co}$, which we may equivalently view as $\bSpan(\Cc,\Cc,\Cc_R)^{\co}\to\bCat_\infty$.
\end{construction}

\begin{corollary}
    This $2$-functor $\bSpan(\Cc,\Cc,\Cc_R)^{\co}\to\bCat_\infty$ classifies the $1$-cartesian fibration \[\bSpan(\Ee,\Ee_R^{p\textup{-cart}},\Ee)^{\co}\to\bSpan(\Cc,\Cc_R,\Cc)^{\co}\] from Proposition~\ref{prop:subtle}. In particular, its underlying functor $\bSpan(\Cc,\Cc,\Cc_R)\to\Cat_\infty$ classifies the cartesian fibration $\Span(\Ee,\Ee_R^{p\textup{-cart}},\Ee)\to\Span(\Cc,\Cc_R,\Cc)$.
    \begin{proof}
        This is clear from the universal property as both restrict to $\phi$ on $\Cc^{\op}$.
    \end{proof}
\end{corollary}

\begin{example}
    Let $T$ be a small category such that its finite coproduct completion $\mathbb F[T]$ has pullbacks (i.e.~$T$ is `orbital'). Product preserving functors $\Cc\colon \mathbb F[T]^{\op}\to\Cat_\infty$ are called \emph{$T$-$\infty$-categories} in \cite{NardinShah}, and we say a $T$-$\infty$-category has \emph{finite $T$-products} if it is right adjointable. In this case, the extension to $\Span(\mathbb F[T],\mathbb F[T],\mathbb F[T])$ obtained via Construction~\ref{constr:nardin-shah-unfurl} is called the \emph{$T$-cartesian $T$-symmetric monoidal structure} on $\Cc$ in \cite[§2.4.1]{NardinShah}. The corollary then gives a simple description of the corresponding cartesian fibration, generalizing the result for trivial $T$ previously proven in \cite[Proposition~3.1.3]{CHLL_Bispans}.\footnote{Note that while the effect of passing to opposite categories pointwise on the unstraightening of a functor can be made explicit \cite{BGN2018Dualizing}, the resulting description would be more complicated and involve iterated span categories.} In the case that $\Cc$ is given by $A\mapsto\mathbb F[T]_{/A}$ with functoriality via pullbacks, this description was conjectured in \cite[Notation~3.4.4]{CHLL_Bispans2}.
\end{example}

\section{A bestiary of unfurling constructions}
Let us summarize all the different unfurling constructions encountered throughout this paper for easy reference:
\begin{equation*}\small\renewcommand{\arraystretch}{2}
\begin{tabular}{r|c|c|}
\multicolumn{1}{l}{} & \multicolumn{1}{c}{left adjointable} & \multicolumn{1}{c}{right adjointable} \\
\cline{2-3}
covariantly     & \textbf{???} & $\Str^\text{cc}\big(\bSpan(\Un^\text{cc}F,\text{cart},\text{all})\big)$\\
\cline{2-3}
contravariantly & $\Str^\text{cc}\big(\bSpan(\Un^\text{ct}F,\text{cart},\text{all})\big)$ & $\Str^\text{ct}\big(\bSpan(\Un^\text{ct}F,\text{cart},\text{all})^\text{co}\big)$\\
\cline{2-3}
\end{tabular}
\end{equation*}
\par\vskip12pt\noindent

We do not know how to describe the extension of a covariantly left adjointable functor $F$ in a similarly nice way: note that the description $\Str^{\ct}(\bSpan(\Un^\cc F,\text{cart},\text{all})^\co)$ one obtains via pattern matching does not make sense as $\Un^\cc(F)$ need not have any cartesian edges. It does however make sense, and is the correct description, if each map $f_!\colon F(A)\to F(B)$ admits a right adjoint $f^*$, by the same argument as in \Cref{prop:subtle}. While we can again embed $F$ into some functor $\hat F$ whose structure maps have right adjoints, the resulting cartesian subfibration of $\bSpan(\Un^\cc \hat F,\text{cart},\text{all})^\co$ does not seem to have an intrinsic description in terms of $\Un^\cc F$.

\bibliography{Bibliography}

\begin{bibdiv}
\begin{biblist}

\bib{Barwick2017SpectralMackey}{article}{
      author={Barwick, Clark},
       title={Spectral {Mackey} functors and equivariant algebraic {{\(K\)}}-theory. {I}.},
        date={2017},
     journal={Adv. Math.},
      volume={304},
       pages={646\ndash 727},
}

\bib{BalmerDellAmbrogio2020Mackey}{book}{
      author={Balmer, Paul},
      author={Dell'Ambrogio, Ivo},
       title={Mackey 2-functors and {Mackey} 2-motives},
      series={EMS Monogr. Math.},
   publisher={European Mathematical Society (EMS)},
        date={2020},
}

\bib{BGN2018Dualizing}{article}{
      author={Barwick, Clark},
      author={Glasman, Saul},
      author={Nardin, Denis},
       title={Dualizing cartesian and cocartesian fibrations},
        date={2018},
     journal={Theory Appl. Categ.},
      volume={33},
       pages={67\ndash 94},
         url={www.tac.mta.ca/tac/volumes/33/4/33-04abs.html},
}

\bib{BachmannHoyois2021Norms}{book}{
      author={Bachmann, Tom},
      author={Hoyois, Marc},
       title={Norms in motivic homotopy theory},
      series={Ast{\'e}risque},
   publisher={Paris: Soci{\'e}t{\'e} Math{\'e}matique de France (SMF)},
        date={2021},
      volume={425},
}

\bib{BenZviNadler}{incollection}{
      author={Ben-Zvi, David},
      author={Nadler, David},
       title={Nonlinear traces},
        date={2021},
   booktitle={Derived algebraic geometry},
   publisher={Paris: Soci{\'e}t{\'e} Math{\'e}matique de France (SMF)},
       pages={39\ndash 84},
}

\bib{CCRY_Characters}{article}{
      author={Carmeli, Schachar},
      author={Cnossen, Bastiaan},
      author={Ramzi, Maxime},
      author={Yanovski, Lior},
       title={Characters and transfer maps via categorified traces},
        date={2022},
     journal={arXiv preprint arXiv:2210.17364},
}

\bib{CisinskiDeglise2019Triangulated}{book}{
      author={Cisinski, Denis-Charles},
      author={D{\'e}glise, Fr{\'e}d{\'e}ric},
       title={Triangulated categories of mixed motives},
      series={Springer Monogr. Math.},
   publisher={Springer},
        date={2019},
}

\bib{CHLL_Bispans}{article}{
      author={Cnossen, Bastiaan},
      author={Haugseng, Rune},
      author={Lenz, Tobias},
      author={Linskens, Sil},
       title={Homotopical commutative rings and bispans},
        date={2024},
     journal={arXiv preprint {arXiv:2403.06911}},
}

\bib{CHLL_Bispans2}{article}{
      author={Cnossen, Bastiaan},
      author={Haugseng, Rune},
      author={Lenz, Tobias},
      author={Linskens, Sil},
       title={Normed equivariant ring spectra and higher {Tambara} functors},
        date={2024},
     journal={arXiv preprint {arXiv:2407.08399}},
}

\bib{CLL_Spans}{article}{
      author={Cnossen, Bastiaan},
      author={Lenz, Tobias},
      author={Linskens, Sil},
       title={Parametrized (higher) semiadditivity and the universality of spans},
        date={2024},
     journal={arXiv preprint {arXiv:2403.07676}},
}

\bib{ElmantoHaugseng2023Distributivity}{article}{
      author={Elmanto, Elden},
      author={Haugseng, Rune},
       title={On distributivity in higher algebra. {I}: {The} universal property of bispans},
        date={2023},
     journal={Compos. Math.},
      volume={159},
      number={11},
       pages={2326\ndash 2415},
}

\bib{EHKSY_Algebraic_Cobordism}{article}{
      author={Elmanto, Elden},
      author={Hoyois, Marc},
      author={Khan, Adeel~A.},
      author={Sosnilo, Vladimir},
      author={Yakerson, Maria},
       title={Modules over algebraic cobordism},
        date={2020},
     journal={Forum Math. Pi},
      volume={8},
       pages={44},
}

\bib{GaitsgoryRozenbluym2017StudyDAG}{book}{
      author={Gaitsgory, Dennis},
      author={Rozenblyum, Nick},
       title={A study in derived algebraic geometry. {Volume} {I}: {Correspondences} and duality},
      series={Math. Surv. Monogr.},
   publisher={Providence, RI: American Mathematical Society (AMS)},
        date={2017},
      volume={221},
}

\bib{harpaz2020ambidexterity}{article}{
      author={Harpaz, Yonatan},
       title={Ambidexterity and the universality of finite spans},
        date={2020},
     journal={Proc. Lond. Math. Soc. (3)},
      volume={121},
      number={5},
       pages={1121\ndash 1170},
}

\bib{HaugsengSpans}{article}{
      author={Haugseng, Rune},
       title={Iterated spans and classical topological field theories},
        date={2018},
     journal={Math. Z.},
      volume={289},
      number={3},
       pages={1427\ndash 1488},
}

\bib{HHLN2022TwoVariable}{article}{
      author={Haugseng, Rune},
      author={Hebestreit, Fabian},
      author={Linskens, Sil},
      author={Nuiten, Joost},
       title={Two-variable fibrations, factorisation systems and {{\(\infty\)}}-categories of spans},
        date={2023},
     journal={Forum Math. Sigma},
      volume={11},
       pages={70},
}

\bib{hopkinsLurie2013ambidexterity}{article}{
      author={Hopkins, Michael},
      author={Lurie, Jacob},
       title={Ambidexterity in {$K(n)$}-local stable homotopy theory},
        date={2013},
     journal={arXiv preprint arXiv:1510.03304},
}

\bib{Kaledin_Mackey_Profunctors}{book}{
      author={Kaledin, D.},
       title={Mackey profunctors},
      series={Mem. Am. Math. Soc.},
   publisher={Providence, RI: American Mathematical Society (AMS)},
        date={2022},
      volume={1385},
}

\bib{Kuijper_Six_Functors}{article}{
      author={Kuijper, Josefien},
       title={A descent principle for compactly supported extensions of functors},
        date={2023},
     journal={Ann. \(K\)-Theory},
      volume={8},
      number={3},
       pages={489\ndash 529},
}

\bib{Lenz-Mackey}{article}{
      author={Lenz, Tobias},
       title={Global algebraic {K}-theory is {Swan} {K}-theory},
         how={Preprint, {arXiv}:2202.07272},
        date={2022},
     journal={arXiv preprint arXiv:2202.07272},
}

\bib{lurie2009HTT}{book}{
      author={Lurie, Jacob},
       title={Higher topos theory},
      series={Ann. Math. Stud.},
   publisher={Princeton, NJ: Princeton University Press},
        date={2009},
      volume={170},
}

\bib{Lurie2009ClassificationTFTs}{incollection}{
      author={Lurie, Jacob},
       title={On the classification of topological field theories},
        date={2009},
   booktitle={Current developments in mathematics, 2008},
   publisher={Somerville, MA: International Press},
       pages={129\ndash 280},
}

\bib{lurie2016HA}{article}{
      author={Lurie, Jacob},
       title={Higher algebra},
        date={2017},
     journal={\url{https://www.math.ias.edu/~lurie/papers/HA.pdf}},
}

\bib{MacPherson2022Bivariant}{article}{
      author={Macpherson, Andrew~W.},
       title={A bivariant {Yoneda} lemma and {{\((\infty, 2)\)}}-categories of correspondences},
        date={2022},
     journal={Algebr. Geom. Topol.},
      volume={22},
      number={6},
       pages={2689\ndash 2774},
}

\bib{mann2022sixFunctor}{article}{
      author={Mann, Lucas},
       title={A $p$-adic 6-functor formalism in rigid-analytic geometry},
        date={2022},
     journal={arXiv preprint arXiv:2206.02022},
}

\bib{marc-n-infty}{article}{
      author={Marc, Gregoire},
       title={{A higher Mackey functor description of algebras over an $N_\infty$-operad}},
        date={2024},
     journal={{arXiv:2402.12447}},
}

\bib{nardin2016exposeIV}{article}{
      author={Nardin, Denis},
       title={{Parametrized higher category theory and higher algebra: Expos\'e IV -- Stability with respect to an orbital $\infty$-category}},
        date={2016},
     journal={arXiv preprint arXiv:1608.07704},
}

\bib{NardinShah}{article}{
      author={Nardin, Denis},
      author={Shah, Jay},
       title={Parametrized and equivariant higher algebra},
         how={Preprint, {arXiv}:2203.00072},
        date={2022},
     journal={arXiv preprint arXiv:2203.00072},
}

\bib{Nuiten2021Straightening}{article}{
      author={Nuiten, Joost},
       title={On straightening for {Segal} spaces},
        date={2024},
     journal={Compos. Math.},
      volume={160},
      number={3},
       pages={586\ndash 656},
}

\bib{Ramzi-Yoneda}{article}{
      author={Ramzi, Maxime},
       title={An elementary proof of the naturality of the {Yoneda} embedding},
        date={2023},
     journal={Proc. Am. Math. Soc.},
      volume={151},
      number={10},
       pages={4163\ndash 4171},
}

\bib{ScholteSixFunctors}{misc}{
      author={Scholze, Peter},
       title={{Six-Functor Formalisms}},
   publisher={{Lecture course at the University of Bonn, \url{https://people.mpim-bonn.mpg.de/scholze/SixFunctors.pdf}}},
        date={2023},
}

\bib{Stefanich2020Correspondences}{article}{
      author={Stefanich, Germán},
       title={{Higher sheaf theory I: Correspondences}},
        date={2020},
     journal={arXiv preprint arXiv:2011.03027},
}

\bib{Voevodsky_Correspondence}{incollection}{
      author={Voevodsky, Vladimir},
       title={Triangulated categories of motives over a field},
        date={2000},
   booktitle={Cycles, transfers, and motivic homology theories},
   publisher={Princeton, NJ: Princeton University Press},
       pages={188\ndash 238},
}

\end{biblist}
\end{bibdiv}
\end{document}